\RequirePackage{xr-hyper}
\documentclass[final,3p,times]{elsarticle}
\makeatletter
\def\ps@pprintTitle{%
	\let\@oddhead\@empty
	\let\@evenhead\@empty
	\def\@oddfoot{\footnotesize\itshape
		\hfill\today}%
	\let\@evenfoot\@oddfoot}

\makeatother

\usepackage{amssymb}
\usepackage{amsthm,amssymb,amsmath,lineno,hyperref,graphicx}
\usepackage{amssymb}
\modulolinenumbers[5]
\usepackage{todonotes,csquotes}
\newtheorem{assumption}{Assumption}

\usepackage{pgfplots,pgfplotstable}
\pgfplotsset{compat=1.18}
\usepackage[capitalize]{cleveref}
\Crefname{assumption}{Assumption}{Assumptions}

\theoremstyle{definition}
\newtheorem{definition}{Definition}[section]
\theoremstyle{plain}
\newtheorem{theorem}[definition]{Theorem}
\newtheorem{lemma}{Lemma}[section]

\newtheorem{remark}[definition]{Remark}
\theoremstyle{definition}

\usepackage{mathtools}
\newcommand{\SDEdim}{\ensuremath{n}}
\DeclareMathOperator{\diff}{d}

\newcommand{\ds}{{\diff}s}
\newcommand{\dt}{{\diff}t}
\newcommand{\du}{{\diff}u}
\newcommand{\dv}{{\diff}v}
\newcommand{\dU}{{\diff}U}
\newcommand{\dB}{{\diff}B}
\newcommand{\tNum}{N}
\newcommand{\tone}{t}
\newcommand{\ttwo}{\tilde{t}}
\newcommand{\eps}{\varepsilon}
\begin{document}
\begin{frontmatter}

\title{Exponential Euler method for stiff stochastic differential equations with additive fractional Brownian noise}


\author[mymainaddress]{Minoo Kamrani\texorpdfstring{$^{*}$}{}}{\corref{mycorrespondingauthor}}
\ead{m.kamrani@razi.ac.ir}
\cortext[mycorrespondingauthor]{Corresponding author}
\author[mysecondaryaddress]{Kristian Debrabant}{}
\ead{debrabant@imada.sdu.dk}

\author[mymainaddress]{Nahid Jamshidi}{}
\ead{n.jamshidi.math@gmail.com}

\address[mymainaddress]{Department of Mathematics, Faculty of Science, Razi University, Kermanshah, Iran}
\address[mysecondaryaddress]{University of Southern Denmark, Department of Mathematics and Computer Science (IMADA), Campusvej 55, 5230 Odense M, Denmark}


\begin{abstract}
We discuss a system of stochastic differential equations with a stiff linear term and additive noise driven by fractional Brownian motions (fBms) with Hurst parameter $ H>\frac{1}{2}$, which arise e.\,g.,\ from spatial approximations of stochastic partial differential equations. For their numerical approximation, we present an exponential Euler scheme and show that it converges in the strong sense with an exact rate close to the Hurst parameter $H$. Further, based on \cite{buckwar11tns}, we conclude the existence of a unique stationary solution of the exponential Euler scheme that is pathwise asymptotically stable.
\end{abstract}


\end{frontmatter}

\section{Introduction}\label{intro}
Fractional Brownian motion (fBm) is a stochastic process that was first introduced within a Hilbert space framework by Kolmogorov  and which generalizes the standard Brownian motion. It is a zero-mean Gaussian process with covariance
\begin{equation*}
Cov(B^H_t,B^H_s)=\mathbb{E}[B^H_t B^H_s]=\frac{1}{2}\left(t^{2H}+s^{2H}-|t-s|^{2H}\right).
\end{equation*}
When $ H=\frac 12$  fBm becomes the standard Brownian motion. However, when $ H\neq \frac 12 $ fBm behaves differently from the standard Brownian motion. For example, for $H\in(0,\frac12)$, fBm has infinite quadratic variation, while for $H\in(\frac12,1)$ its quadratic variation vanishes. Moreover, fBm is neither a semimartingale nor a Markov process, so the powerful tools from the classical theories of stochastic analysis are not applicable.

FBms have many applications in models arising in engineering, physics, telecommunication networks, climate science and finance, where the systems exhibit randomness and long-range dependence \cite{Abry,Mandelbrot}.

Since many stochastic differential equations (SDEs) driven by fBm cannot be solved explicitly, their numerical approximation is important. Various approximation schemes have been introduced for SDEs with fBm   \cite{ hong2020optimal,kamrani2015numerical, MR2456334, MR2474352}. For efficient approximation, schemes that can reproduce the essential dynamics of the exact solution are of interest. In the frequent situation that the essential dynamics can be represented by linear terms, methods that treat these terms explicitly / solve linear differential equations exactly are well discussed for ordinary differential equations, see, e.g.\ \cite{cox02etd, maday90aoi} and references therein, as well as for stochastic differential equations driven by Brownian motion, see, e.g., \cite{carbonell08wll, jimenez99sos} rsp.\ \cite{arara19sbs, debrabant21rkl,debrabant22lsf,erdogan19anc,tambue16wcf,yang19sps} and references therein for local linearization methods rsp.\ exponential integrators.

In this article, we consider the numerical approximation of a system of SDEs with a linear stiff term and additive noise driven by fBm of the form
\begin{equation}\label{s1}
 \begin{split}
 \dU_t&=\underbrace{\left(AU_t+f(t,U_t)\right)}_{=:g(t,U_t)}\dt+\sum_{i=1}^m b_i(t) \dB^{H}_i(t),~~~t\in[t_0,T],\\
 {U}_{t_0}&=u_0,~~~ u_0\in\mathbb{R}^\SDEdim,
 \end{split}
 \end{equation}
where the stochastic processes $ B^{H}_i,  i=1,\dots,m,$ are independent two-sided fBms  (i.\,e.,\ defined for all $ t\in \mathbb{R} $) with Hurst parameter $ \frac{1}{2}<H<1 $ (the quadratic variation vanishes), $ f\in C([t_0,T]\times\mathbb{R}^\SDEdim,\mathbb{R}^\SDEdim) $,  $A\in\mathbb{R}^{\SDEdim,\SDEdim}$ and $ b_i:[t_0,T]\to \mathbb{R}^{\SDEdim}$ for $i=1, \dots ,m$, $\SDEdim\in\mathbb{N}$. Systems of the form \eqref{s1} arise, for example, as semidiscretizations of stochastic diffusion-reaction systems of the form $\du=(u_{xx}+f(t,u))\dt+b(t,x)\dB^{H}$, one of the simplest examples being the stochastic heat equation with additive, long-range dependent noise.
 \Cref{s1} is considered as a pathwise Riemann-Stieltjes integral equation. Note that the existence of a unique stationary and attracting solution for SDE \eqref{s1} with Hurst parameter $ H> \frac{1}{2} $ follows from \cite{MR2095071}.
As $f$ can also contain linear parts, the choice of $A$ and $f$ in representation \eqref{s1} is not unique for given $g$.

 We are interested in the case that the drift term $g(t,U_t)$ in \eqref{s1} is stiff, and $A$ and $f$ can be chosen such that only the linear part $ AU_t $ is stiff, i.\,e.,\ $|A|(T-t_0)\gg1$ and $\mu[A](T-t_0)\ll|A|(T-t_0)$ where $ |\cdot| $ denotes the Euclidean norm and $\mu(\cdot)$ is the corresponding logarithmic matrix norm, see e.g.\ \cite{dekker84sor,strehmel2012numerik}, while th{e nonlinear part $ f(t,U_t) $ is non-stiff. This type of problem often arises from spatial approximations of stochastic partial differential equations by Galerkin, finite difference, or finite element approximations \cite{MR3554794, MR3651090}, where $A$ could e.\,g.\ correspond to the discretized differential operator in space. Note however that the concrete choice of $A$ has significant influence on the long time integration error.
  In the situation that there is no noise, exponential integrators have been proven to be a very interesting class of numerical time integration methods \cite{hochbruck10ei}, as they handle the stiff part exactly and can thus be stable even when being explicit.

In this paper, we will extend one of the most straightforward exponential schemes, the exponential Euler scheme, to the SDE driven by fractional Brownian motion \eqref{s1}.
In contrast to other explicit schemes that are applicable to \eqref{s1}, like the (standard) Euler method (\cite{neuenkirch06oao}) or Milstein type schemes (\cite{nourdin05sda}), the exponential Euler method does not suffer from severe step size restrictions when applied to stiff SDEs, which can make the aforementioned methods in practice unusable. The price for this is that one needs to calculate matrix exponentials and simulate the appearing stochastic integrals (see \cref{pre}). \label{ref:1}

We will prove that by avoiding the presence of $|A|$ in any error bound, the (stiff) convergence rate of the exponential Euler scheme can be shown to equal the Hurst parameter $H$.

  In the second part of the paper, the pathwise stability properties of the established scheme are considered. There are very few papers dealing specifically with the stability of numerical schemes for nonlinear SDEs driven by fBm. Some related papers look at a similar issue for nonlinear SDEs driven by Brownian motion; see \cite{caraballo06tpn,MR2543841}. These papers discuss the implicit Euler-Maruyama method and find that it has a unique stochastic stationary solution that attracts all other solutions pathwise in the pullback and the forward sense.

 In \cite{MR2524684}, the authors proved that under a one-sided dissipative condition, both the stochastic system \eqref{s1} with $ A\equiv 0 $ and any Hurst parameter $ H\in(0,1)$ and its discretization with the drift-implicit Euler method  have a unique stationary solution, which pathwise attracts all other solutions. Moreover, they showed that the latter converges to the first one as the step size tends to zero.

  Buckwar et al.\ \cite{buckwar11tns} considered for the Brownian motion case with $ H=\frac{1}{2} $ the stability properties of the $ \theta $-Maruyama method, a linear implicit and two exponential Euler schemes. They proved for each of these methods, there exists a unique stochastic stationary solution that is pathwise asymptotically stable. As discussed in Section \ref{stability}, their proof carries forward to the stability of the exponential Euler scheme applied to the semi-linear SDE \eqref{s1} driven by fractional Brownian motion with Hurst parameter $ \frac{1}{2}<H<1 $.

 The remaining part of this article is structured as follows: In Section \ref{pre} we introduce the exponential Euler method and make some assumptions on the function $ f(t,U_t) $ and the matrix $ A $ in SDE \eqref{s1}. Convergence results are established in Section \ref{convergence}. Section \ref{stability} is devoted to the pathwise stability analysis of our numerical method, and finally, in Section \ref{numeric} we present a numerical example.
\section{Setting of the problem}\label{pre}
This section contains some assumptions on the SDE \eqref{s1} and the numerical method.
The solution of \cref{s1} can be represented as \cite{MR1912142, MR2095071}
 \begin{equation}\label{s2}
 U_t=e^{A(t-t_0)}u_0+\int_{t_0}^te^{A(t-s)}f(s,U_s)\ds+\sum_{i=1}^m \int_{t_0}^te^{A(t-s)}b_i(s)\dB^{H}_i(s),\qquad t\in[t_0,T],
 \end{equation}
 where $U_{t_0}=u_0$ and the integral with respect to $ B^H $ is a pathwise Riemann-Stieltjes integral  \cite{MR2493996, MR2150381}.

Let $\phi(u,v)=H(2H-1)|u-v|^{2H-2}, u,v>0,$ and  $f:\mathbb R^+ \to \mathbb R$ be a Borel measurable (deterministic) function. The
 function $f$ belongs to the Hilbert space $L^2_{\phi}(\mathbb R^+)$
 if
 $$|f|^2_{\phi}=\int_{0}^{\infty}\int_{0}^{\infty} f(u)f(v) \phi(u,v) \du \dv <\infty.$$
 By the following lemma the stochastic integral with respect to fractional Brownian motions for deterministic kernels is  defined.
 \begin{lemma}\label{lemmacov}\cite{MR1741154}
If $f,g \in L^2_{\phi}(\mathbb R^+),$ then $\int_0^{\infty} f(u) \dB^H(u)$ and $\int_0^{\infty} g(u) \dB^H(u)$
 are well defined
 zero mean, Gaussian random variables such that
 \begin{equation}
\mathbb E\Big(\int_0^{\infty} f(u) dB^H(u) \int_0^{\infty} g(u) \dB^H(u) \Big)=\int_0^{\infty} \int_0^{\infty} f(u) g(v)\phi(u,v) \du \dv.
 \end{equation}
 \end{lemma}

 For an integer $ N$ and the given time interval $ [t_0,T] $ we introduce a time discretization with time points $ t_0<t_1<\dots<t_\tNum=T  $, with possibly non-equidistant step sizes $ h_i=t_{i+1}-t_{i} $, $ 0\leq i \leq N-1 $, and denote $ h_{max}=\max_{0\leq i\leq N-1}h_i, h_{min}=\min_{0\leq i\leq N-1}h_i  $. \\
Assume $A$ is regular, then by generalizing the exponential Euler scheme for SDEs driven by Brownian motion \cite{jentzen09oto} to \eqref{s2}, we obtain for $k=0,1,\dots, \tNum-1$
 \begin{equation}\label{s3}
 V_{k+1}=e^{Ah_k}V_k+A^{-1}\left(e^{Ah_k}-I\right)f(t_k,V_k)+\sum_{i=1}^m \int_{t_k}^{t_{k+1}}e^{A(t_{k+1}-s)}b_i(s)\dB^{H}_i(s),
  \end{equation}
where $V_k$ is approximation to $U(t_k)$ for $k=1,\dots,N$ and $ V_0=u_0 $.
From Lemma \ref{lemmacov}, for $i=1,\dots,m$ and $k=0,1,\dots, \tNum-1$,
\[ I_{i,k}= \int_{t_k}^{t_{k+1}}e^{A(t_{k+1}-s)}b_i(s)\dB^{H}_i(s),\]
 are well-defined zero mean Gaussian random variables with covariance
  \begin{equation}\label{fBm}
  \begin{split}
  \mathbb{E}\left(I_{i,k}I_{j,l}^\top\right)=\delta_{i,j}H(2H-1)\int_{t_k}^{t_{k+1}}\int_{t_l}^{t_{l+1}}e^{A(t_{k+1}-u)}b_i(u)\left(e^{A(t_{l+1}-v)}b_j(v)\right)^\top|u-v|^{2{H}-2}\dv\du,
  \end{split}
  \end{equation}
  where $ \delta_{i,j}, i,j=1,\dots,m $ denotes the Kronecker delta.
  A Cholesky decomposition of this covariance matrix allows therefore to simulate \eqref{s3}.

  In the following, $ |.| $ denotes the Euclidean norm, while $ \|.\|=\left(\mathbb{E}[|.|^2]\right)^{\frac{1}{2}} $. Further, we will use the logarithmic matrix norm (that actually is not a norm, but a sublinear functional) defined by
  \begin{equation}\label{eq:mudef}
  \mu[B]=\lambda_{max}\left(\frac{B+B^T}{2}\right),
  \end{equation}
  for any square matrix $ B $. The logarithmic matrix norm satisfies \cite{dekker84sor} $ |\mu[B]|\leq |B| $ and
  \begin{equation} \label{boundExp}
  \left|\exp(Bt)\right|\leq \exp(\mu[B]t)~~\text{ for }~t\geq 0.
  \end{equation}
  For stiff problems, $|A|$ is typically very large, while $\mu[A]$ is of small absolute value and often negative. In the following, we will therefore ensure that all estimates are suitable for this type of problem, and thus avoid any estimate involving $|A|$.

   We will use the following assumptions on the functions $f$ and $b_i$ and the matrix $A$:
   \begin{assumption}\label{ass:f} The nonlinear function $f:~[t_0,T]\times\mathbb{R}^\SDEdim\to\mathbb{R}^\SDEdim$ is continuous and has linear growth condition, i.\,e.,\ there is a positive constant $ D $ such that for all $t \in [t_0,T]$ and $ x\in \mathbb{R}^\SDEdim $
   	\begin{equation}\label{bounded}
   		|f(t,x)|\leq D(1+|x|).
   	\end{equation}
   	 Further, $f$ is Lipschitz continuous, i.\,e.,\ there exists a positive constant $ K $  such that
  	\begin{equation}\label{lip}
  	|f(t,x)-f(s,y)|\leq K(|t-s|+|x-y|),\quad \forall t,s \in [t_0,T] , \quad  \forall x,y\in \mathbb{R}^\SDEdim.
  	\end{equation}
  \end{assumption}
  \begin{assumption}\label{ass:b}
  	The functions $b_i:[t_0,T]\to \mathbb{R}^\SDEdim , i=1,\dots,m$ are bounded with respect to the $|.|$, i.\,e.,
  	\begin{equation}\label{boundb}
  |b_i(t)|^2\leq M ~ \text{ for } t\in [t_0,T], i=1,\dots,m.
  	\end{equation}
  	 \end{assumption}
   \begin{assumption}\label{ass:A}
    There is a positive constant $L$ such that $A\in\mathbb{R}^{\SDEdim,\SDEdim}$ for all $t>s$ satisfies the following conditions
    \[
    |Ae^{A(t-s)}|\leq\frac{L}{t-s},\qquad |A^{-1}(I-e^{A(t-s)})|\leq L(t-s).
    \]
    Further, $\sup_{t_0\leq t\leq T}e^{\mu[A](t-t_0)}\leq C$ with a constant $C$ not too large.
   \end{assumption}
   Note that in applications, often $\mu[A]\leq0$, which then implies that $\sup_{t_0\leq t\leq T}e^{\mu[A](t-t_0)}\leq 1$.

A finite dimensional version of the following lemma can be used to show that \cref{ass:A} is fulfilled.
\begin{lemma}\label{lemmalord}\cite{MR3308418}
  Suppose that $H$ is a Hilbert space with inner product $\langle .,.\rangle$ and the
   linear operator $-A:D(A)\subset H \to H$ has a complete orthonormal set of eigenfunctions $\{\phi_j :j\in \mathbb N\}$
   and eigenvalues $\lambda_j>0$, ordered so that $\lambda_{j+1}\geq \lambda_j.$
   Then for $t\geq 0$, the
   exponential of $-tA$ which is defined by
   $$ e^{-tA}u=\sum_{j=1}^{\infty} e^{-\lambda_j t} \langle u,\phi_j\rangle \phi_j,$$
   is a semigroup of linear
   operators on $H.$ Therefore
   for each $\alpha\geq 0$, there exists a constant $K$ such that
      \begin{equation} \label{semi1}
   \|A^{\alpha} e^{-tA}\|_{L(H)}\leq K t^{-\alpha},~~t>0,
   \end{equation}
  and for $\alpha\in [0,1]$
   \begin{equation} \label{semi2}
  \|A^{-\alpha} (I-e^{-tA})\|_{L(H)}\leq K t^{\alpha}~~ t\geq 0.
   \end{equation}
   \end{lemma}
  \section{Convergence rate of the exponential Euler method}\label{convergence}
  This section is devoted to the convergence of the proposed numerical method \eqref{s3}. We first prove the following lemma.
\begin{lemma}\label{lemma}
Assume that\cref{ass:f,ass:b}  are satisfied and $\mathbb{E}|u_0|^2< \infty$. Then $$\sup_{0\leq t \leq T}\mathbb{E}|U_t|^2< \infty,$$ where $U_t$ is the solution of \eqref{s2}.
\end{lemma}
\begin{proof}
	By applying the inequality $ (a+b+c)^2\leq 3(a^2+b^2+c^2) $ and the Cauchy-Schwarz inequality, we obtain
	\begin{equation}\label{eq:13}
	\begin{split}
	|U_t|^2&\leq 3 \left\{|e^{A(t-t_0)}u_0|^2+\big|\int_{t_0}^te^{A(t-s)}f(s,U_s)\ds\big|^2+\big|\sum_{i=1}^m\int_{t_0}^te^{A(t-s)}b_i(s)\dB^{H}_i(s)\big|^2\right\}\\
	&\leq 3\left\{|e^{A(t-t_0)}|^2|u_0|^2+\int_{t_0}^t|e^{A(t-s)}|^2 \ds\int_{t_0}^t|f(s,U_s)|^2 \ds+\big|\sum_{i=1}^m\int_{t_0}^te^{A(t-s)}b_i(s)\dB^{H}_i(s)\big|^2\right\}.
	\end{split}
	\end{equation}
	From (\ref{boundExp}) we conclude
	 \begin{equation} \label{boundExp1}
	 \left|e^{A(t-s)}\right|^2\leq e^{2\mu[A](t-s)} ~~\text{ for }~t\geq s\geq t_0,
	 \end{equation}
	 therefore, by \cref{ass:f} and the inequality $(1+|U_s|)^2\leq 2(1+|U_s|^2)$ we get
	 \begin{equation}\label{eq:15}
	 \begin{split}
	 |U_t|^2&\leq 3 \left\{ e^{2\mu[A](t-t_0)}|u_0|^2+\int_{t_0}^t  e^{2\mu[A](t-s)} \ds\int_{t_0}^t2D^2\left(1+|U_s|^2\right) \ds+\big|\sum_{i=1}^m\int_{t_0}^te^{A(t-s)}b_i(s)\dB^{H}_i(s)\big|^2\right\}.
	 \end{split}
	 \end{equation}
By taking expectation and using the isometry property of fractional stochastic integrals\cref{lemmacov} we conclude
\begin{equation}
\begin{split}
\mathbb{E}|U_t|^2&
\begin{multlined}[t]
	\leq 3\left\{e^{2\mu[A](t-t_0)}\mathbb{E}|u_0|^2+2 D^2 \sup_{t_0\leq t \leq T}e^{2\mu[A](t-t_0)} \int_{t_0}^{T} \ds  \int_{t_0}^t \mathbb E \left(1+|U_s|^2\right) \ds               \right.\\+\left.\sum_{i=1}^m\int_{t_0}^t\int_{t_0}^t\left|\left(e^{A(t-u)}b_i(u)\right)^\top \left(e^{A(t-v)}b_i(v)\right)\right||u-v|^{2{H}-2}\du\dv\right\}
\end{multlined}\\&
\begin{multlined}[t]
\leq 3\left\{e^{2\mu[A](t-t_0)}\mathbb{E}|u_0|^2+2 D^2 \sup_{t_0\leq t \leq T}e^{2\mu[A](t-t_0)}) (T-t_0) \Big((T-t_0)+\int_{t_0}^t \mathbb{E} |U_s|^2\ds\Big)  \right.\\+\left.\sum_{i=1}^m\int_{t_0}^t\int_{t_0}^t\left|\left(e^{A(t-u)}b_i(u)\right)^\top \left(e^{A(t-v)}b_i(v)\right)\right||u-v|^{2{H}-2}\du\dv\right\}.
\end{multlined}
\end{split}
\end{equation}
Finally, applying \cref{ass:b} we obtain
\begin{equation}
  \mathbb{E}|U_t|^2\leq 3(\sup_{t_0\leq t \leq T}e^{2\mu[A](t-t_0)})\left\{ \mathbb{E}|u_0|^2+2 D^2 (T-t_0)^2 +2 D^2 (T-t_0) \int_{t_0}^t \mathbb{E} |U_s|^2\ds+mM(T-t_0)^{2H}\right\},
\end{equation}
and by using the Gronwall lemma, we get the result.
\end{proof}
\begin{lemma}\label{FracInt}\cite{MMVA}
Assume function $f$ is such that $\int_0^T |f(u)|^{\frac 1 H}\du$ is finite. Then there exists a positive constant $C_1$ such that
\begin{equation}
\int_0^T\int_0^T f(u)f(v)|u-v|^{2H-2} \du\dv \leq C_1 \Big(\int_0^T |f(u)|^{\frac 1 H}\du\Big)^{2H}.
\end{equation}
\end{lemma}
We establish the following theorem, which will be used to prove the convergence rate of the exponential Euler method  \eqref{s3}.
\begin{theorem}\label{Reqularity}
Assume that \cref{ass:f,ass:A} are fulfilled and $\mathbb{E}|u_0|^2< \infty$ as well as $\mathbb{E}|Au_0|^2<\infty$.
	Then for every $\eps>0,$ there exists a constant $ C>0 $ such that
 	\begin{equation}
 	\left(\mathbb{E}|U_{\tone}-U_{\ttwo}|^2\right)^{\frac{1}{2}}\leq C(\ttwo-\tone)^{H-\eps},
 	\end{equation}
    for all $\tone, \ttwo \in [t_0,T] $, $\tone<\ttwo$, where $ U_{t} $ is the solution of SDE \eqref{s1} at time $ t$.
 \end{theorem}
 \begin{proof}
 Throughout the proof let C be a constant that may depend on  $t_0$, $T$, $|u_0|$, $D$, $L$, $K$, $H$, $\mu(A)$, $C_1$, $\sum_{i=1}^m |b_i|$, $\mathbb{E}|Au_0|^2,$ $\eps$ and may change from line to line . By \eqref{s2} and the H\"older inequality, it holds
 	\begin{equation}
	\begin{split}
 	\mathbb{E}|U_{\tone}-U_{\ttwo}|^2
 	=&\mathbb{E}\Big|e^{A(\tone-t_0)}u_0+\int_{t_0}^{\tone}e^{A(\tone-s)}f(s,U_s)ds
 	+\sum_{i=1}^{m}\int_{t_0}^{\tone}e^{A(\tone-s)}  b_i(s)\dB_i^H(s)\\&-e^{A(\ttwo-t_0)}u_0-\int_{t_0}^{\ttwo}e^{A(\ttwo-s)}f(s,U_s)ds
 	-\sum_{i=1}^{m}\int_{t_0}^{\ttwo}e^{A(\ttwo-s)} b_i(s)\dB_i^H(s)\Big|^2
 	\\& \leq 5\Big(\mathbb{E}\Big|(e^{A(\tone-t_0)}-e^{A(\ttwo-t_0)})u_0\Big|^2 +\mathbb{E}\Big|\int_{t_0}^{\tone}\big(e^{A(\tone-s)}-e^{A(\ttwo-s)}\big)f(s,U_s)ds\Big|^2+\mathbb{E}\Big|\int_{\tone}^{\ttwo} e^{A(\ttwo-s)} f(s,U_s)ds\Big|^2\\&
 	+\mathbb{E}\Big|\sum_{i=1}^{m}\int_{t_0}^{\tone}\big(e^{A(\tone-s)}-e^{A(\ttwo-s)}\big) b_i(s)\dB_i^H(s)\Big|^2
 	+\mathbb{E}\Big|\sum_{i=1}^{m}\int_{\tone}^{\ttwo}e^{A(\ttwo-s)} b_i(s)\dB_i^H(s)\Big|^2 \Big)
 	 \\&
 	 =: 5 (E_0+E_1+E_2+E_3+E_4).
 	\end{split}
 		\end{equation}
For the first term, i.e.,\ $E_0$, by \cref{ass:A} we have
  \begin{equation}\label{E0}
  \begin{split}
  E_0&=\mathbb{E}\Big|(e^{A(\tone-t_0)}-e^{A(\ttwo-t_0)})u_0\Big|^2
  =
  \mathbb{E}\Big|e^{A(\tone-t_0)}(I-e^{A(\ttwo-\tone)})u_0\Big|^2\\&
  \leq  (\sup_{t_0 \leq t\leq T} e^{2\mu[A] (t-t_0)})\Big|A^{-1}(I-e^{A(\ttwo-\tone)})\Big|^2 \mathbb{E}|Au_0|^2\\&
  \leq C (\ttwo-\tone)^2.
  \end{split}
  \end{equation}
 		For $E_1$, by first taking the norm of the integrand and then using the H\"older inequality, \cref{ass:f,ass:A,lemma} we obtain
 		\begin{equation}\label{E1}
 			\begin{split}
 		 	E_1=&\mathbb{E}\Big|\int_{t_0}^{\tone}\big(e^{A(\tone-s)}-e^{A(\ttwo-s)}\big)f(s,U_s)ds\Big|^2\\&
 		 	\leq \int_{t_0}^{\tone}\big| e^{A(\tone-s)}-e^{A(\ttwo-s)}\big|\mathbb{E}|f(s,U_s)|^2\ds\int_{t_0}^{\tone}\big| e^{A(\tone-s)}-e^{A(\ttwo-s)}\big|\ds\\&
 		 	\leq C\int_{t_0}^{\tone}\big| e^{A(\tone-s)}-e^{A(\ttwo-s)}\big|(1+\mathbb{E}|U_s|^2)\ds\int_{t_0}^{\tone}\big| e^{A(\tone-s)}-e^{A(\ttwo-s)}\big|\ds\\&
 		 	\leq C\Big(\int_{t_0}^{\tone}\big| e^{A(\tone-s)}-e^{A(\ttwo-s)}\big|\ds\Big)^2\\&
\leq C \Big(\int_{t_0}^{\tone} \big| e^{A(\tone-s)}-e^{A(\ttwo-s)}\big|^{\eps} \big| e^{A(\tone-s)}-e^{A(\ttwo-s)}\big|^{1-\eps}  \ds\Big)^2\\&
 		 	\leq C(\sup_{t_0 \leq t\leq T} e^{2\eps \mu[A] (t-t_0)}) \Big(\int_{t_0}^{\tone} \big| e^{A(\tone-s)}-e^{A(\ttwo-s)}\big|^{1-\eps} \ds\Big)^2
 		 	\\& \leq  C \Big( \int_{t_0}^{\tone} \Big| A e^{A(\tone-s)}\Big|^{1-\eps} \Big
 		 	|A^{-1}\Big(I-e^{A(\ttwo-\tone)}\Big) \Big|^{1-\eps}\ds\Big)^2
 		 	\\& \leq  C \Big(\int_{t_0}^{\tone} (\tone-s)^{-1+\eps}(\ttwo-\tone)^{1-\eps}\ds\Big)^2
 		 	\\& \leq C  (\ttwo-\tone)^{2-2\eps}.
  		 	 		 \end{split}
 		 		\end{equation}
 		 		For $E_2$, by using  \cref{ass:f,lemma} we get
 		 		\begin{equation}\label{E2}
 		 		 			\begin{split}
 		 		 		 	E_2=&\mathbb{E}\Big|\int_{\tone}^{\ttwo} e^{A(\ttwo-s)}f(s,U_s)\ds\Big|^2\\&
 		 		 		 	\leq \mathbb{E} \Big(\int_{\tone}^{\ttwo}\big| e^{A(\ttwo-s)}\big| |f(s,U_s)|\ds\Big)^2\\&
 		 		 		 		\leq \mathbb{E} \Big(\int_{\tone}^{\ttwo}\big| e^{A(\ttwo-s)}\big|^2 \ds. \int_{\tone}^{\ttwo}| f(s,U_s)|^2\ds\Big)\\&
 		 		 		 	\leq 2(\ttwo-\tone) (\sup_{t_0 \leq t\leq T} e^{2 \mu[A] (t-t_0)})D^2 \int_{\tone}^{\ttwo} (1+\mathbb{E}|U_s|^2)\ds\\&
 		 		 		  \leq C (\ttwo-\tone)^2.
 		 		 		 	 	\end{split}
 		 		 		 		\end{equation}
For estimating $E_3$, by independence property of the fractional Brownian motions for $i\neq j,$ we have
\[
\mathbb{E} \left(\int_{t_0}^{\tone}\big(e^{A(\tone-s)}-e^{A(\ttwo-s)}\big) b_i(s)\dB_i^H(s) \int_{t_0}^{\tone}\big(e^{A(\tone-s)}-e^{A(\ttwo-s)}\big) b_j(s)\dB_j^H(s)\right)=0.
\]
Therefore by  H\"older inequality and  analogue to (\ref{fBm}), we can deduce
\begin{equation}
 		\begin{split}
 	E_3&=\mathbb{E}\Big|\sum_{i=1}^{m}\int_{t_0}^{\tone}\big(e^{A(\tone-s)}-e^{A(\ttwo-s)}\big) b_i(s)\dB_i^H(s)\Big|^2\\&
  =\sum_{i=1}^{m} \mathbb{E}\Big| \int_{t_0}^{\tone}\big(e^{A(\tone-s)}-e^{A(\ttwo-s)}\big) b_i(s)\dB_i^H(s)\Big|^2
 \\&  =\sum_{i=1}^{m} \int_{t_0}^{\tone} \int_{t_0}^{\tone} \Big| \Big(\big(e^{A(\tone-u)}-e^{A(\ttwo-u)}\big) b_i(u)\Big)^T \Big(\big(e^{A(\tone-v)}-e^{A(\ttwo-v)}\big) b_i(v)\Big)\Big||u-v|^{2H-2}\du\dv\\
 &\leq \sum_{i=1}^{m}\int_{t_0}^{\tone} \int_{t_0}^{\tone} \big|\big(e^{A(\tone-u)}-e^{A(\ttwo-u)}\big)\big| \big|\big(e^{A(\tone-v)}-e^{A(\ttwo-v)}\big)\big||b_i(u)||b_i(v)||u-v|^{2H-2}\du\dv.
   		 		 		 		 		\end{split}
 		 		 		 		\end{equation}
Similar to (\ref{E1}), by Assumption \ref{ass:b} we deduce
 		 		 		 	\begin{equation}\label{E33}
 		 		 		 		 		\begin{split}
 		 		 		 		 &E_3\\ &\leq mM   \int_{t_0}^{\tone} \int_{t_0}^{\tone} \Big| e^{A(\tone-u)}- e^{A(\ttwo-u)}\Big|^{1-(H-\eps)}
 		 		 		 		\Big| e^{A(\tone-u)}- e^{A(\ttwo-u)}\Big|^{H-\eps}
 		 		 		 		 \Big| e^{A(\tone-v)}- e^{A(\ttwo-v)}\Big|^{1-(H-\eps)}
 		 		 		\Big| e^{A(\tone-v)}- e^{A(\ttwo-v)}\Big|^{H-\eps}
 	 		 		\Big||u-v|^{2H-2}\du\dv \\& \leq
 		 		 		 		mM  (\sup_{t_0 \leq t\leq T} e^{2(1-(H-\eps)) \mu[A] (t-t_0)})\int_{t_0}^{\tone} \int_{t_0}^{\tone} \Big|A e^{A(\tone-u)}\Big|^{H-\eps} \Big|A^{-1}\Big(I-e^{A(\ttwo-\tone)} \Big)\Big| ^{H-\eps}
 		 		 	 		\Big|A e^{A(\tone-v)}\Big|^{H-\eps} \Big|A^{-1}\Big(I-e^{A(\ttwo-\tone)} \Big)\Big|^{H-\eps}
	 		 		 		\Big||u-v|^{2H-2}\du\dv.
 		 		 		 		 \end{split}
 		 		 		 	   	\end{equation}
 		 		 		 	   	Consequently,
 		 		\begin{equation}\label{E13}
 		 		 		 	   	 		\begin{split}
 		 		 		 	  	& E_3 \leq
 		 		 		 	   		mM(\sup_{t_0 \leq t\leq T} e^{(1+\eps) \mu[A] (t-t_0)})(\ttwo-\tone)^{2(H-\eps)} \int_{t_0}^{\tone} \int_{t_0}^{\tone} (\tone-u)^{\eps-H}
 		 		 		 	   	 		 		 	 	(\tone-v)^{\eps-H}
 		 		 		 	   		 		 		 		|u-v|^{2H-2}\du\dv,
 		 		 		 	   		 		 		  \end{split}
 		 		 		 	   	 		 		  	\end{equation}
 hence by Lemma \ref{FracInt} it follows that
 \begin{equation}\label{E3}
  		  E_3 \leq C (\ttwo -\tone)^{2(H-\eps)}.
 	 	   	\end{equation}
 Finally, for the last term by (\ref{fBm}) we have
\begin{equation}\label{E4}
 		 		\begin{split}
 		 		 		 		 	E_4=&\mathbb{E}\Big|\sum_{i=1}^{m}\int_{\tone}^{\ttwo}e^{A(\ttwo-s)} b_i(s)\dB_i^H(s)\Big|^2\\&
 		 		 		 		  		 		   = \sum_{i=1}^{m} \int_{\tone}^{\ttwo} \int_{\tone}^{\ttwo} \Big| \Big(e^{A(\ttwo-u)} b_i(u)\Big)^T \Big(e^{A(\ttwo-v)} b_i(v)\Big)\Big||u-v|^{2H-2}\du\dv
 		 		 		 		  		 		  \\& \leq  mM (\sup_{t_0 \leq t\leq T} e^{2 \mu[A] (t-t_0)})  \int_{\tone}^{\ttwo} \int_{\tone}^{\ttwo} |u-v|^{2H-2} \du\dv
 		 		 		 		  		 		    		  \\&  \leq C (\ttwo-\tone)^{2H}.
 		 		 		 		  		 		    	\end{split}
 		 		 		 		 		\end{equation}
 		From (\ref{E0}), (\ref{E1}), (\ref{E2}), (\ref{E3}) and (\ref{E4}) we conclude
 		 		 		 		 	\begin{equation}
 		 		 		 		 	 	\left(\mathbb{E}|U_{\tone}-U_{\ttwo}|^2\right)^{\frac{1}{2}}\leq C(\ttwo-\tone)^{H-\eps}.
 		 		 		 		 	 	\end{equation}
 		 		 		 		 	 	 		 		 		 		 	 \end{proof}
 \begin{theorem}\label{th:Convergence}
 Assume  $\mathbb{E}|u_0|^2< \infty$, $\mathbb{E}|Au_0|^2<\infty$ and  \cref{ass:f,ass:b,ass:A} are fulfilled.
 Then for every $\eps>0,$ there exists a constant $ C>0 $ such that
 	\begin{equation}\label{Conver}
 	\sup_{k=0, \dots, \tNum}\sqrt{ \mathbb{E}|U_{t_k}-V_k|^2}\leq C h_{max}^{H-\eps},
 	\end{equation}
    for all $ \tNum\geq 2 $, where $ U_{t_k} $ is the solution of SDE \eqref{s1} at time $ t_k$, and $ V_k $ is the numerical solution given by \eqref{s3} for $ k=0,1, \dots, \tNum $.
 \end{theorem}
 \begin{proof}
  \Cref{s3} implies
  \begin{equation}
  V_{k+1}=e^{A(t_{k+1}-t_0)}u_0+\sum_{l=0}^{k}\int_{t_l}^{t_{l+1}}e^{A(t_{k+1}-s)}f(t_l,V_l)\ds+\sum_{i=1}^m\int_0^{t_{k+1}}e^{A(t_{k+1}-s)}b_i(s)\dB^H_i(s).
  \end{equation}
Let
\begin{equation}
X_{k+1}= e^{A(t_{k+1}-t_0)}u_0+\sum_{l=0}^{k}\int_{t_l}^{t_{l+1}}e^{A(t_{k+1}-s)}f(t_l,U_{t_l})\ds+\sum_{i=1}^m\int_0^{t_{k+1}}e^{A(t_{k+1}-s)}b_i(s)\dB^H_i(s).
\end{equation} 	
By the triangle inequality, we obtain
\begin{equation}\label{triangle}
\|V_{k+1}-U_{t_{k+1}}\|\leq \|V_{k+1}-X_{k+1}\|+\|X_{k+1}-U_{t_{k+1}}\|.
\end{equation} 	
For the first term of the right-hand side of \eqref{triangle}, by Lipschitz condition \eqref{lip} we can write
\begin{equation}\label{firsterror}
\begin{split}
\|V_{k+1}-X_{k+1}\|&=\left\|\sum_{l=0}^{k}\int_{t_l}^{t_{l+1}}e^{A(t_{k+1}-s)}[f(t_l,V_l)-f(t_l,U_{t_l})]\ds\right\|\\
&\leq \sum_{l=0}^{k}\int_{t_l}^{t_{l+1}}|e^{A(t_{k+1}-s)}|\left\|f(t_l,V_l)-f(t_l,U_{t_l})\right\|\ds\\
&\leq \sup_{t_0\leq t\leq T}e^{\mu[A](t-t_0)}K\sum_{l=0}^{k}h_l\|V_l-U_{t_l}\|.
\end{split}
\end{equation}	
To estimate the second term of the right hand side of \eqref{triangle}, i.\,e.,\ $ \|X_{k+1}-U_{t_{k+1}}\| $, by H\"older inequality we have
\begin{equation*}
\begin{split}
 \|X_{k+1}-U_{t_{k+1}}\|^2&=\mathbb{E}|X_{k+1}-U_{t_{k+1}}|^2=\mathbb{E}\left|\sum_{l=0}^k\int_{t_l}^{t_{l+1}}e^{A(t_{k+1}-s)}[f(s,U_s)-f(t_l,U_{t_l})]\ds\right|^2
 \\& \leq (k+1)\sum_{l=0}^k \mathbb{E} \left|\int_{t_l}^{t_{l+1}}e^{A(t_{k+1}-s)}[f(s,U_s)-f(t_l,U_{t_l})]\ds\right|^2
 \\& \leq  (k+1) \sum_{l=0}^k h_l  \left[\int_{t_l}^{t_{l+1}}\big|e^{A(t_{k+1}-s)}\big|^2\mathbb{E}\big|f(s,U_s)-f(t_l,U_{t_l})\big|^2\ds\right].
\end{split}
\end{equation*}
Therefore \cref{ass:f,Reqularity} yield
\begin{equation}\label{seconderror}
\begin{split}
 \|X_{k+1}-U_{t_{k+1}}\|^2&
 \leq  C(k+1) \sup_{t_0\leq t\leq T}e^{2\mu[A](t-t_0)} \sum_{l=0}^k h_l  \left[\int_{t_l}^{t_{l+1}}\big[(s-t_l)^2+\mathbb{E}|U_s-U_{t_l}|^2\big]\ds\right].
 \\&
 \leq  C (k+1) \sum_{l=0}^k h_l  \left[\int_{t_l}^{t_{l+1}}\big[(s-t_l)^2+(s-t_l)^{2(H-\eps)}\big]\ds\right]
 \\&
 \leq  C (k+1)  \sum_{l=0}^k  h_l^{2+2H-2\eps}.
  \end{split}
\end{equation}
From (\ref{firsterror}) and (\ref{seconderror}) we conclude
\begin{equation}
\begin{split}
 \|V_{t_{k+1}}-U_{t_{k+1}}\|^2&
 \leq  C\sum_{l=0}^{k}h_l\|V_l-U_{t_l}\|^2+ C(k+1) \sum_{l=0}^k  h_l^{2+2H-2\eps}.
   \end{split}
\end{equation}
Thus by Gronwall lemma, we get
\begin{equation}
 \|V_{t_{k+1}}-U_{t_{k+1}}\|^2=\mathbb E|V_{t_{k+1}}-U_{t_{k+1}}|^2 \leq C(k+1) \sum_{l=0}^k   h_l^{2+2H-2\eps}\leq Ch_{\max}^{2H-2\eps} \sum_{l=0}^kh_l  \leq CT h_{\max}^{2H-2\eps}.
\end{equation}
Therefore
\begin{equation}
	\sup_{k=0, \dots, \tNum}\sqrt{ \mathbb{E}|U_{t_k}-V_k|^2}\leq C h_{max}^{H-\eps}.
\end{equation}
	 \end{proof}
\begin{remark}
At first glance, the condition $\mathbb{E}|Au_0|^2<\infty$ in \cref{th:Convergence} could be seen as quite restrictive due to $Au$ being stiff. However, in a situation where $A$ is a discretized differential operator in space, it will typically be fulfilled for sufficiently differentiable initial conditions.
\end{remark}
\section{Pathwise stability}\label{stability}
In this section, based on \cite{cheridito2003fractional, MR2524684}, we establish a brief summary of the fractional Ornstein-Uhlenbeck process. Pathwise stability of the exponential Euler scheme applied to SDEs with a Brownian motion has been investigated in \cite{buckwar11tns}. By extending that result we conclude the stability of the numerical scheme \eqref{s3} applied to the \SDEdim-dimensional SDE \eqref{s1} with m-dimensional fractional Brownian motion $ B^H(t) $.
\subsection{Fractional Ornstein-Uhlenbeck process}
Consider the scalar linear SDE
\begin{equation}
dX(t)=-\alpha X(t)dt+\sum_{i=1}^m b_i(t) \dB^{H}_{i}(t),
\end{equation}
where $  \alpha> 0 $. From \eqref{s2} we know that for a given random initial value $X_0 $ at $t_0$ it has
the explicit solution
\begin{equation}\label{o1}
X(t)=e^{-\alpha (t-t_0)}X_0+e^{-\alpha t}\sum_{i=1}^m \int_{t_0}^t e^{\alpha s}b_i(s)dB^H(s), \quad \text{almost surely}.
\end{equation}
The existence of the pathwise Riemann– Stieltjes integral  also follows from \cite{cheridito2003fractional}.
Taking the pathwise pullback limit $t_0\to-\infty$, we get
\begin{equation}\label{fou}
\hat{X}(t)= e^{-\alpha t}\sum_{i=1}^m \int_{-\infty}^{t}e^{\alpha s}b_i(s)d{B}^H_s, \quad \text{almost surely},
\end{equation}
which is called  a fractional Ornstein-Uhlenbeck process. By Lemma 1 in \cite{MR2524684}, this process is well defined.
$ \left(\hat{X}(t)\right)_{t\in\mathbb{R}} $ is a Gaussian process, and by stationarity property of the increments of the fractional Brownian motion, it is stationary \cite{cheridito2003fractional}. Further, with $X_0=\hat{X}(t_0)$, it is the solution of \eqref{o1}. Thus, the fractional Ornstein-Uhlenbeck process \eqref{fou} is a stochastic stationary solution of the linear test SDE \eqref{o1}.

Let $ X^1(t) $ and $ X^2(t) $ be any two solutions of SDE \eqref{o1}. Then it holds
\begin{equation}
X^1(t)-X^2(t)=e^{-\alpha (t-t_0)}(X_0^1-X_0^2)\rightarrow 0, \quad as~ t\rightarrow \infty, \quad \text{almost surely}.
\end{equation}
Replacing $ X^2(t) $ by the fractional Ornstein-Uhlenbeck process $ \hat{X}(t)  $, we see that it attracts all other solutions of SDE \eqref{o1} forward in time in the pathwise sense.
\subsection{Pathwise stability for exponential Euler method}
In this subsection, we are interested in the stability of the exponential Euler method \eqref{s3} approximating semi-linear SDEs of the form \eqref{s1}.
For simplicity of notation we assume $ h_0=h_1=\dots =h_{\tNum-1}=h $.

The following theorem, which was proven by \cite{buckwar11tns} for SDEs driven by Brownian motion, also extends to method \eqref{s3} approximating semi-linear SDEs of the form \eqref{s1}.
\begin{theorem}\label{th:4.1}
	Suppose that $ \mu[A] \leq 0  $ with respect to a suitable induced matrix norm, and $f$ satisfies the Lipschitz condition \eqref{lip} with Lipschitz constant $K$ and that
	\begin{equation}\label{condition}
	K|A||A^{-1}|<-\mu[A].
	\end{equation}
Then there exists  $ h^{\ast} > 0 $ such that the exponential Euler scheme \eqref{s3} has a unique stochastic stationary solution which is pathwise asymptotically stable for all $ h \in (0, h^\ast)$. In particular, such an $ h^\ast $ is given by the positive solution of
	\begin{equation}
	1+hK|A^{-1}||A|e^{h|A|-h\mu[A]}=e^{-h\mu[A]}.
	\end{equation}
\end{theorem}
\begin{proof}
The proof follows exactly from \cite[Proof of Theorem 3.2]{buckwar11tns}, with the only difference being that the Ornstein-Uhlenbeck process is replaced by its fractional generalization \eqref{fou}.
\end{proof}
\section{Numerical experiment}\label{numeric}
We conclude this article with
 a numerical example.
 Consider the following stiff nonlinear system
 \begin{align*}
 dU_t&=\left(EU_t+\sin(U_t)\right)dt+b dB^H_t,\quad  t\in [0,1],\\
 U_0&=\sqrt{\frac2{n+1}}(\sin\frac{\pi}{n+1},\sin\frac{2\pi}{n+1},\dots,\sin\frac{n\pi}{n+1})^\top,
 \end{align*}
 where $ B^H(t) $ is a fractional Brownian motion with Hurst parameter $ \frac{1}{2}<H<1 $, $b$ is an $ n-$ dimensional vector given by
 \[
 b=
[1,~1,\dots,1]^\top
 \] and	$ E\in\mathbb{R}^{\SDEdim\times\SDEdim} $ is given by
 	\[
 	E= (n+1)^2\begin{bmatrix}
 	-2 & 1 & 0&0&\dots &0\\
 	1 & -2 &  1&0&\dots &0\\
 	0 & \ddots & \ddots & \ddots& \ddots&\vdots\\
 	\vdots & \ddots & \ddots & \ddots& \ddots& \vdots\\
 	0& \dots&0&1 &-2&1\\
 	0& \dots& \dots&0&1 &-2
 	\end{bmatrix}.
 	\]
 	Note that
 	\[|U_0|^2= \frac{2}{n+1}\sum_{k=1}^{n}\sin^2(\frac{k\pi}{n+1})=
 	\frac1{n+1}\sum_{k=0}^{n}\Re\left(1-e^{i2\frac{k\pi}{n+1}}\right)
 	= 1.\]
 	The matrix $-E$ has an orthonormal set of eigenfunctions, with eigenvalues
 	\begin{equation}
 	\lambda_k=(n+1)^2\Big(2-2\cos(\frac{k\pi}{n+1})\Big)>0,~k=1,\cdots,n
 	\end{equation}
 	where $\lambda_{k+1}\geq\lambda_k.$ Therefore, by a finite dimensional version of Lemma \ref{lemmalord} we conclude that \cref{ass:A} is satisfied.

    In the following we choose $n=100$ and apply the exponential Euler method with $ 16, 32, 64, 128$ and $256$ steps. \Cref{fig2} shows the root mean square errors obtained by averaging over $ 1000 $ independent fractional Brownian paths over $ [0,1] $ for $H=0.6, 0.7, 0.8$ and $0.9$. As a reference solution, we used here a numerical approximation with $N=2048$ steps. From these results, the orders of convergence $p$ of the method can be approximated as slopes of the regression lines in \cref{fig2} for the different $H$-values. The results indicate an order of convergence of 1, that is higher than what we proved in \cref{th:Convergence}, so might indicate that our theoretical result can be improved. As the result of the current proof basically is tight to the smoothness of the exact solution, this would require a completely different approach.

   Note that the classical Euler method could not be applied in this example, as due to its inherent instability the approximations \enquote{explode}, see \cref{fig:Euler}.

  \begin{figure}[!htp]
  	\begin{center}
  		\begin{tikzpicture}
  		\begin{axis}[
  		width=0.5\linewidth,
  		xmode=log,
  		log basis x=2,
  		log basis y=2,
  		ymode=log,
  		legend pos=south east,
  		xlabel=$h$
  		]
  		\pgfplotstableread{
  			xrow    yrow6    yrow7   yrow8   yrow9
        0.003906250000000   0.000086285679103   0.000045671006230   0.000025425724125   0.000015979470772
        0.007812500000000   0.000186727477482   0.000098711289640   0.000055323657431   0.000034739453129
        0.015625000000000   0.000391279542106   0.000207070005502   0.000116560145200   0.000073643376789
        0.031250000000000   0.000804347113296   0.000439108526317   0.000248115774387   0.000156296113494
        0.062500000000000   0.001683132275344   0.000940795871605   0.000537832319485   0.000356504920526
  		}\datatable
  		\addplot+[only marks,blue] table[x=xrow,y=yrow6] {\datatable};
  		\addplot[mark=none, blue,forget plot] table[x=xrow, y={create col/linear regression={y=yrow6}}]{\datatable};
  		\xdef\slopesix{\pgfplotstableregressiona}
  		\addlegendentry{$H=0.6$, $p=\pgfmathprintnumber{\slopesix}$}
  		\addplot+[only marks,blue] table[x=xrow,y=yrow7] {\datatable};
  		\addplot[mark=none, blue,forget plot] table[x=xrow, y={create col/linear regression={y=yrow7}}]{\datatable};
  		\xdef\slopeseven{\pgfplotstableregressiona}
  		\addlegendentry{$H=0.7$, $p=\pgfmathprintnumber{\slopeseven}$}
  		\addplot+[only marks,red] table[x=xrow,y=yrow8] {\datatable};
  		\addplot[mark=none, red,forget plot] table[x=xrow, y={create col/linear regression={y=yrow8}}]{\datatable};
  		\xdef\slopeeight{\pgfplotstableregressiona}
  		\addlegendentry{$H=0.8$, $p=\pgfmathprintnumber{\slopeeight}$}
  		\addplot+[only marks,green] table[x=xrow,y=yrow9] {\datatable};
  		\addplot[mark=none, green,forget plot] table[x=xrow, y={create col/linear regression={y=yrow9}}]{\datatable};
  		\xdef\slopenine{\pgfplotstableregressiona}
  		\addlegendentry{$H=0.9$, $p=\pgfmathprintnumber{\slopenine}$}
  		\end{axis}
  		\end{tikzpicture}\end{center}
  	\caption{Mean square errors of the exponential Euler method for different Hurst parameters $H$. $p$ denotes the respective numerically approximated
convergence order.} \label{fig2}
  \end{figure}
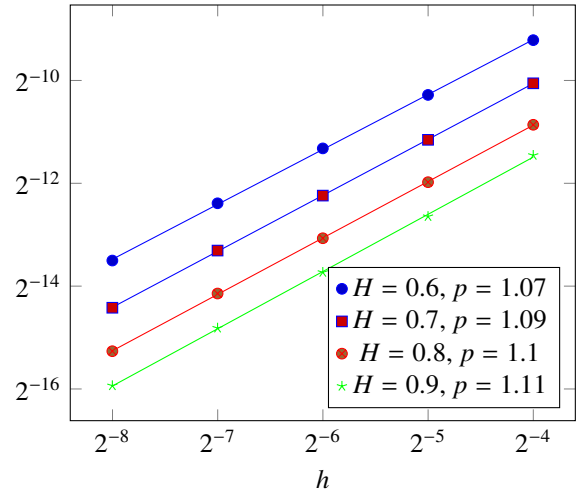

  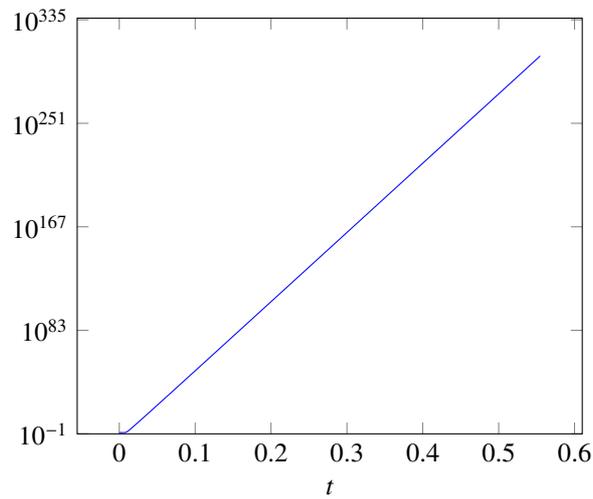
\begin{figure}[!htp]
  	\begin{center}
  		\begin{tikzpicture}
  		\begin{axis}[
  		width=0.5\linewidth,
  		ymode=log,
  		legend pos=south east,
  		xlabel=$t$,
        ymin=0.1
  		]
  		\pgfplotstableread{
  			t    v
     0.0000000e+00   1.0000000e+00
   3.9062500e-03   9.9710217e-01
   7.8125000e-03   1.0160481e+00
   1.1718750e-02   1.9279113e+01
   1.5625000e-02   2.2158478e+03
   1.9531250e-02   2.8073042e+05
   2.3437500e-02   3.7444758e+07
   2.7343750e-02   5.1575932e+09
   3.1250000e-02   7.2619248e+11
   3.5156250e-02   1.0389724e+14
   3.9062500e-02   1.5046633e+16
   4.2968750e-02   2.2000267e+18
   4.6875000e-02   3.2416623e+20
   5.0781250e-02   4.8069400e+22
   5.4687500e-02   7.1661187e+24
   5.8593750e-02   1.0731669e+27
   6.2500000e-02   1.6134092e+29
   6.6406250e-02   2.4338584e+31
   7.0312500e-02   3.6824725e+33
   7.4218750e-02   5.5863510e+35
   7.8125000e-02   8.4944804e+37
   8.2031250e-02   1.2943747e+40
   8.5937500e-02   1.9760958e+42
   8.9843750e-02   3.0220637e+44
   9.3750000e-02   4.6289234e+46
   9.7656250e-02   7.1003359e+48
   1.0156250e-01   1.0905592e+51
   1.0546875e-01   1.6770500e+53
   1.0937500e-01   2.5818395e+55
   1.1328125e-01   3.9789036e+57
   1.1718750e-01   6.1378569e+59
   1.2109375e-01   9.4767782e+61
   1.2500000e-01   1.4644332e+64
   1.2890625e-01   2.2647490e+66
   1.3281250e-01   3.5050250e+68
   1.3671875e-01   5.4282957e+70
   1.4062500e-01   8.4123953e+72
   1.4453125e-01   1.3044985e+75
   1.4843750e-01   2.0240469e+77
   1.5234375e-01   3.1422231e+79
   1.5625000e-01   4.8806818e+81
   1.6015625e-01   7.5847183e+83
   1.6406250e-01   1.1792428e+86
   1.6796875e-01   1.8342643e+88
   1.7187500e-01   2.8543440e+90
   1.7578125e-01   4.4435274e+92
   1.7968750e-01   6.9201989e+94
   1.8359375e-01   1.0781297e+97
   1.8750000e-01   1.6802665e+99
   1.9140625e-01  2.6195914e+101
   1.9531250e-01  4.0853663e+103
   1.9921875e-01  6.3733062e+105
   2.0312500e-01  9.9455655e+107
   2.0703125e-01  1.5524585e+110
   2.1093750e-01  2.4239939e+112
   2.1484375e-01  3.7858162e+114
   2.1875000e-01  5.9142507e+116
   2.2265625e-01  9.2416193e+118
   2.2656250e-01  1.4444441e+121
   2.3046875e-01  2.2581566e+123
   2.3437500e-01  3.5310560e+125
   2.3828125e-01  5.5226701e+127
   2.4218750e-01  8.6394170e+129
   2.4609375e-01  1.3517851e+132
   2.5000000e-01  2.1155141e+134
   2.5390625e-01  3.3113613e+136
   2.5781250e-01  5.1841437e+138
   2.6171875e-01  8.1175481e+140
   2.6562500e-01  1.2712991e+143
   2.6953125e-01  1.9913310e+145
   2.7343750e-01  3.1196785e+147
   2.7734375e-01  4.8881540e+149
   2.8125000e-01  7.6603147e+151
   2.8515625e-01  1.2006410e+154
   2.8906250e-01  1.8821004e+156
   2.9296875e-01  2.9507587e+158
   2.9687500e-01  4.6268380e+160
   3.0078125e-01  7.2559282e+162
   3.0468750e-01  1.1380418e+165
   3.0859375e-01  1.7851659e+167
   3.1250000e-01  2.8006100e+169
   3.1640625e-01  4.3941920e+171
   3.2031250e-01  6.8953525e+173
   3.2421875e-01  1.0821406e+176
   3.2812500e-01  1.6984764e+178
   3.3203125e-01  2.6661388e+180
   3.3593750e-01  4.1855475e+182
   3.3984375e-01  6.5715376e+184
   3.4375000e-01  1.0318721e+187
   3.4765625e-01  1.6204212e+189
   3.5156250e-01  2.5449083e+191
   3.5546875e-01  3.9972156e+193
   3.5937500e-01  6.2788966e+195
   3.6328125e-01  9.8638970e+197
   3.6718750e-01  1.5497167e+200
   3.7109375e-01  2.4349712e+202
   3.7500000e-01  3.8262409e+204
   3.7890625e-01  6.0129415e+206
   3.8281250e-01  9.4501146e+208
   3.8671875e-01  1.4853263e+211
   3.9062500e-01  2.3347515e+213
   3.9453125e-01  3.6702256e+215
   3.9843750e-01  5.7700226e+217
   4.0234375e-01  9.0718168e+219
   4.0625000e-01  1.4264035e+222
   4.1015625e-01  2.2429589e+224
   4.1406250e-01  3.5272026e+226
   4.1796875e-01  5.5471397e+228
   4.2187500e-01  8.7244258e+230
   4.2578125e-01  1.3722496e+233
   4.2968750e-01  2.1585264e+235
   4.3359375e-01  3.3955419e+237
   4.3750000e-01  5.3418021e+239
   4.4140625e-01  8.4041343e+241
   4.4531250e-01  1.3222823e+244
   4.4921875e-01  2.0805635e+246
   4.5312500e-01  3.2738808e+248
   4.5703125e-01  5.1519238e+250
   4.6093750e-01  8.1077482e+252
   4.6484375e-01  1.2760125e+255
   4.6875000e-01  2.0083205e+257
   4.7265625e-01  3.1610704e+259
   4.7656250e-01  4.9757435e+261
   4.8046875e-01  7.8325664e+263
   4.8437500e-01  1.2330257e+266
   4.8828125e-01  1.9411617e+268
   4.9218750e-01  3.0561350e+270
   4.9609375e-01  4.8117632e+272
   5.0000000e-01  7.5762885e+274
   5.0390625e-01  1.1929686e+277
   5.0781250e-01  1.8785442e+279
   5.1171875e-01  2.9582402e+281
   5.1562500e-01  4.6587001e+283
   5.1953125e-01  7.3369422e+285
   5.2343750e-01  1.1555378e+288
   5.2734375e-01  1.8200012e+290
   5.3125000e-01  2.8666678e+292
   5.3515625e-01  4.5154491e+294
   5.3906250e-01  7.1128260e+296
   5.4296875e-01  1.1204715e+299
   5.4687500e-01  1.7651294e+301
   5.5078125e-01  2.7807963e+303
   5.5468750e-01  4.3810524e+305
}\datatable
  		\addplot+[mark=none] table[x=t,y=v] {\datatable};
    		\end{axis}
  		\end{tikzpicture}\end{center}
  	\caption{Mean of 1000 simulated paths of the standard Euler method with $H=0.6$ and step length 1/256.} \label{fig:Euler}
\end{figure}

\end{document}